\documentclass[10pt]{amsart}

\usepackage{amssymb, amsthm, amsmath, graphicx, yfonts, wasysym, appendix, url, verbatim}
\usepackage[margin=4cm]{geometry}
\usepackage{rotating}
\usepackage{amsbsy,enumerate}
\usepackage{graphicx}
\usepackage{ccaption}
\usepackage{xcolor}
\usepackage{times}
\usepackage{wrapfig}
\usepackage{amsmath, amssymb, amsthm} 

\newcommand{\IP}[2]{\left< #1 , #2 \right>}

\newcommand{\R}{\ensuremath{\mathbb{R}}}

\allowdisplaybreaks[4]

\newtheorem{thm}{Theorem}[section]

\newtheorem{lem}[thm]{Lemma}

\newtheorem*{uthm}{Theorem}

\newtheorem{defn}[thm]{Definition}
\newtheorem*{udefn}{Definition}

\theoremstyle{remark}
\newtheorem{rmk}{Remark}

\title{Mean curvature flow with free boundary - Type 2 singularities}
\author{Glen Wheeler and Valentina-Mira Wheeler$^*$}
\thanks{*: Corresponding author.}
\address{Glen Wheeler\\
           Institute for Mathematics and its Applications \\
           University of Wollongong\\
           Northfields Avenue\\
           Wollongong, NSW, 2522, Australia\\
           email: glenw@uow.edu.au }
\address{ Valentina-Mira Wheeler \\
           Institute for Mathematics and its Applications \\
           University of Wollongong\\
           Northfields Avenue\\
           Wollongong, NSW, 2522, Australia\\
           email: vwheeler@uow.edu.au
           }

\begin{document}

\begin{abstract}
In this paper we give sufficient conditions that guarantee the mean curvature
flow with free boundary on an embedded rotationally symmetric double cone
develops a Type 2 curvature singularity.
We additionally prove that Type 0 singularities may only occur at infinity.
\end{abstract}

\maketitle
\section{Introduction}
We say that a smooth one-parameter family of immersed disks
$F:D^n\times[0,T)\rightarrow\R^{n+1}$ evolves by the mean curvature flow with
free boundary on a support hypersurface $F_\Sigma:\Sigma\rightarrow\R^{n+1}$ if
\begin{align}
\frac{\partial F}{\partial t} = \vec{H} = -H\nu\quad\quad&\text{ on }D^n\times[0,T)
\notag\\
\IP{\nu}{\nu_\Sigma} = 0\quad\quad\quad\quad&\text{ on }\partial D^n\times[0,T)\,,
\label{MCFwFB}\\
F(\partial D^n,t) \subset F_\Sigma(\Sigma)\,,\quad\text{and}&\quad
F(\cdot,0) = F_0(\cdot)\,.
\notag
\end{align}

Local existence follows, as demonstrated by Stahl \cite{thesisstahl}, by
writing the evolving hypersurfaces as graphs for a short time over their
initial data.
Stahl additionally gave continuation criteria: a-priori bounds on the second
fundamental form are sufficient for the global existence of a solution
\cite{stahl1996res,stahl1996convergence}.
In this work he also showed that initially convex data remains convex when the
support hypersurface is umbilic, and that in this situation the flow contracts
to a round hemishperical point (a Type 1 singularity). A generalisation to
other contact angles of Stahl's continuation criteria was obtained by Freire
\cite{freire2010mean}.

Buckland studied a setting similar to that of Stahl, and focused on obtaining a
classification of singularities according to topology and type
\cite{buckland2005mcf}.
Koeller has generalised the regularity theory developed by Ecker and Huisken
\cite{ecker2004rtm,ecker1989mce,ecker1991ieh} to the setting of free boundaries
\cite{koeller2007singular}. His main regularity theorem is a criterion under
which the singular set will has measure zero.

The authors have studied initially graphical mean curvature flow with free
boundary, obtaining long time existence results and results on the formation of
curvature singularities on the free boundary
\cite{vmwheeler2012rotsym,wheeler2014mean,wheeler2014meanhyperplane}.
A similar angle approach has been employed by Lambert
\cite{lambert2014perpendicular} in his work.
Edelen's work is the first systematic treatment of Type 2 singularities
\cite{edelen2014convexity}.
Convexity estimates play a fundamental role in his work.
Edelen's work implies that rescaling a mean curvature flow with free boundary
at a type 2 singularity yields a weakly convex mean curvature flow with free
boundary in a hyperplane.
After reflection a convex translating soliton is obtained which furthermore
decomposes into a product of a strictly convex $k$-manifold and $\R^{n-k}$.
This result implies that solutions satisfying the hypotheses of Theorem
\ref{thmsingularities} also under rescaling have the same structure property.
A family of examples exhibiting this behaviour is given in Remark \ref{rmk1},
see also Figure \ref{Fig1}.

In previous work the second author has used results on the mean curvature flow
of embedded discs inside generalised cylinders to answer questions of existence
and uniqueness of minimal hypersurfaces \cite{wheeler2016mean}. There,
non-existence is proved by establishing that the flow terminates in a curvature
singularity developing at the apex of a pinching cylinder.
A pinching cylinder is defined as follows.

\begin{defn}
Let $\omega_\Sigma:Oz\rightarrow[0,\infty)$ be a continuous function.
Assume that $\omega_\Sigma$ is smooth outside finitely many points $P =
\{w_1,\ldots,w_{n_p}\}$, where $\omega_\Sigma(w_i) = 0$; that is,
$\omega_\Sigma\in C^\infty_{loc}(Oz\setminus P)$.
Assume that there exists a compact set $K \supset P$ such that
\begin{equation*}
z\frac{d\omega_\Sigma}{dz}(z) > 0\quad\text{ for all }z\in Oz\setminus K\,.
\end{equation*}
The function $\omega_\Sigma$ generates a smooth rotationally symmetric disconnected hypersurface
$F_\Sigma:\Sigma\rightarrow\R^{n+1}$, where $\Sigma$ is the disjoint union of $n_p+1$ cylinders.
We term the support hypersurface $F_\Sigma$ a \emph{pinching cylinder}.
\label{pinchingcylinder}
\end{defn}

One result that guarantees the development of a finite-time singularity is:

\begin{thm}[Flow in pinching cylinders \cite{wheeler2016mean}]
Let $\Sigma$ be a pinching cylinder as in Definition \ref{pinchingcylinder} with $n_p=1$.
Let $w_1 = z^* = 0$.
Assume that for all $z\in Oz$
\begin{align}
\IP{{\nu}_{\Sigma}(z)}{e_1} > C_\Sigma \geq 0
\label{Sigma_graph}
\end{align}
where $C_\Sigma$ is a global constant and ${\nu}_{\Sigma}$ is the normal to
${\omega}_{\Sigma}$.
The graph condition \eqref{Sigma_graph} is understood as limits from above and
below at points in $P$.

Suppose that for all $z\in Oz\setminus\{0\}$,
\begin{equation}
\label{conelike}
z\frac{d\omega_\Sigma}{dz}(z) > 0\,.
\end{equation}
Then the maximal time $T$ of existence for any graphical mean curvature flow
$\omega:D(t)\times[0,T)\rightarrow\R$ with free boundary on $\Sigma$ (see
\eqref{Neumannproblem}) is finite.
The hypersurfaces $F:D^n\times[0,T)\rightarrow\R^{n+1}$ generated by $\omega$ contract as $t\rightarrow T$ to the point
$(0,0)$.
\label{thmsingularities}
\end{thm}

In \cite{wheeler2016mean} sufficient conditions were given that guarantee the
singularity will be Type 1 and Type 0.  Type 0 singularities are a new notion
introduced in that paper, where the flow becomes singular purely by the domain
vanishing, with the second fundamental form remaining uniformly bounded.

This paper is concerned with an extension to further classification of the
singularity and more importantly with the proof of existence of a Type 2
singularity for mean curvature flow with free boundaries.

\begin{thm}[Type 2 singularities]
Let $\omega_\Sigma$ and $\omega_0$ be as in Theorem \ref{thmsingularities}.
If there exist three constants $C_1,C_2,C_3\in\R$ such that $0<C_1<\infty$,
$C_1<C_2<\infty$, $0<C_3<\infty$, and two constants $\alpha,
\delta\in(0,\infty)$ satisfying $\frac{2\delta}{\alpha+1}>1$ such that for $z$
sufficiently close to $0$ we have:
\begin{align}
\frac{C_1}{z^\delta} \leq \bigg|\frac{\frac{d\omega_\Sigma}{dz}(z)}{\omega_\Sigma(z)}\bigg| \leq \frac{C_2}{z^\alpha}, \quad\text{ and }\quad
\bigg|\frac{d\omega_\Sigma}{dz}(z)\bigg|\leq C_3\,,
\label{T2condn}
\end{align}
then the singularity from Theorem \ref{thmsingularities} is Type 2, in
particular there exists $C\in(0,\infty)$ such that for $t$ sufficiently close
to $T$ we have
\[
|A|^2(x,t)
\ge
\frac{C}{(T-t)^{\frac{2\delta}{\alpha+1}}}\,.
\]
\label{thmtypeII}
\end{thm}

\begin{rmk}[Examples of $\Sigma$ that produce a Type 2 singularity]
	\label{rmk1}
Let us consider a family of pinching cylinders with profile
\[
	\omega_{\Sigma,k}(z)=\exp\Big(-\frac{1}{z^k}\Big)\,,\quad k > 0\,.
\]
Then
\[
\bigg|\frac{\frac{d\omega_\Sigma}{dz}(z)}{\omega_\Sigma(z)}\bigg|
 = \frac{k}{z^{k+1}}\,,
\]
so \eqref{T2condn} is satisfied with $\alpha = \delta = k+1$ and $C_1 = C_2 =
k$.
The second fundamental form along a graphical mean curvature flow with $\omega_\Sigma = \omega_{k,\Sigma}$ blows up
quickly, with the estimate
\[
|A|^2(x,t)
\ge
\frac{C}{(T-t)^{\frac{2\delta}{\alpha+1}}}
\ge \frac{C}{(T-t)^{1 + \frac{k}{k+2}}}
\,,
\]
for $t$ sufficiently close to $T$.
Interestingly, the rate of blowup for the second fundamental form is never as
fast as $(T-t)^{-2}$, but can be made arbitrarily close.
Figure \ref{Fig1} illustrates this example.
\end{rmk}

\begin{figure}
\includegraphics[trim=4cm 10cm 4cm 10cm,clip=true,width=8cm,height=6cm]{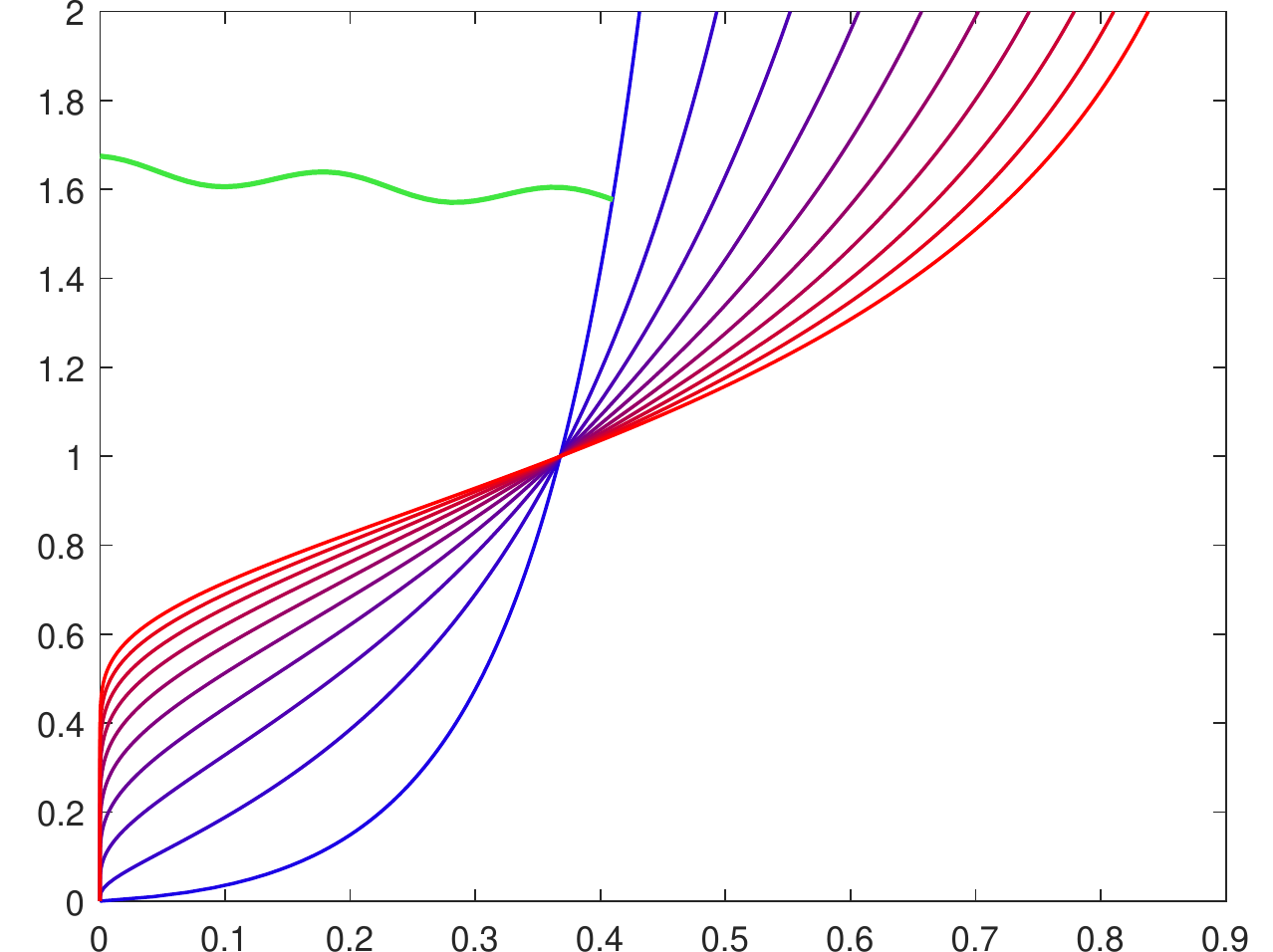}
\caption{Profile curves for $\omega_{k,\Sigma}$ for $k=\frac{l}4$,
$l=1,\ldots,10$. The colours move from blue to red as $l$ increases.
Sample initial data that moves to a type 2 singularity is in green.
}
\label{Fig1}
\end{figure}

We finally show in the following theorem that Type 0 singularities may only
occur at inifnity, indicating that there may not be many more cases of Type 0
singularities than those already found in \cite{wheeler2016mean}.

\begin{thm}[Type 0 singularities]
Let $\omega_\Sigma$ and $\omega_0$ be as in Theorem \ref{thmsingularities}.
The singularity from Theorem \ref{thmsingularities} is not Type 0.
\label{thmtype02}
\end{thm}

This paper is organised as follows. Sections two and three contain definitions
and prerequisites. In section four we discuss previous results for context and
proceed with the proof of Theorem \ref{thmtypeII} and Theorem \ref{thmtype02}.

\section{Mean curvature flow with free boundary supported on a generalised
cylinder}

The behaviour of immersions flowing by the mean curvature flow with free
boundary is largely unknown, with available results in the literature
indicating that a complete picture of asymptotic behaviour irrespective of
initial condition is extremely difficult to obtain \cite{stahl1996res,
koeller2007singular}.
Therefore the relevant question is: under which initial conditions is it
possible to obtain a complete picture of asymptotic behaviour?

Working in the class of graphical hypersurfaces is a viable strategy, so long
as the graph condition can be preserved
\cite{lambert2012constant,thesisvulcanov,wheeler2014meanhyperplane,wheeler2014mean}.
In each of these works, global results were enabled by symmetry of the initial
data and/or of the boundary.
Without such symmetries, recent work indicates that graphicality is not in
general preserved \cite{andrewswheeler} (even in the case where
$F_\Sigma(\Sigma)$ is a standard round sphere).

Let us formally set the support hypersurface
$F_\Sigma:\Sigma\rightarrow\R^{n+1}$ to be rotationally symmetric and generated
by the graph of a function $\omega_\Sigma:Oz\rightarrow\R$ over the $Oz$ axis.

By convention we let $x=(x_1,\ldots,x_n)$ be a point in $\R^n\subset\R^{n+1}$,
with $n\geq 2$ and denote by $y=|x|$ the length of $x$.
With this convention the profile curve of the support surface lies in a plane generated by $Oy$ and $Oz$ axes.
We write the graph condition on $\omega_\Sigma$ as
\begin{align*}
\IP{{\nu}_{\Sigma}(z)}{e_1} > C_\Sigma \geq 0,
\end{align*}
where $C_\Sigma$ is a global constant, ${\nu}_{\Sigma}$ the normal to
${\omega}_{\Sigma}$, and $\IP{\cdot}{\cdot}$ is the standard inner
product in ${\R}^{n+1}$.
Our convention is that $\nu_\Sigma$ points away from the interior of the
evolving hypersurface.

Let us now describe how a rotationally symmetric graphical mean curvature flow
with free boundary $F:D^n\times[0,T)\rightarrow\R^{n+1}$ satisfying
\eqref{MCFwFB} can be represented by the evolution of a scalar function (the
graph function).
Let us set $D(t)=(0,r(t))\subset\R$.
The Neumann boundary is at $\partial D(t)=r(t)$.
The left-hand endpoint of $D(t)$, the zero, is not a true boundary point.
It arises from the fact that the scalar generates a radially symmetric graph
that is topologically a disk.
The coordinate system degenerates at the origin and so it is artificially
introduced as a boundary point.
This is however a technicality, and no issues arise in dealing with quantities
at this fake boundary point, since by symmetry and smoothness we have that
the radially symmetric graph is horizontal at the origin.

We represent the mean curvature flow of a radially symmetric graph
$F:D^n\times[0,T)\rightarrow\R^{n+1}$
by the evolution of its graph function $\omega: D(t)\times
[0,T)\rightarrow \R$ that must satisfy the following:
\begin{align}
\frac{\partial \omega}{\partial t}   &=  \frac{d^2\omega}{dy^2}\
\frac{1}{1+(\frac{d\omega}{dy})^2}+\frac{d\omega}{dy}\
\frac{n-1}{y}&&~~\text{ on }~~(0,r(t))\times[0,T),
\label{Neumannproblem}\\
\IP{{\nu}_\omega}{{\nu}_{\Sigma}} &= 0 \text{ and
}r(t)={\omega}_{\Sigma}(\omega(r(t),t))&&~~\text{ on }~~
r(t)\times[0,T),\notag\\
\lim_{y\rightarrow0}\frac{1}{y}&\frac{d\omega}{dy}(y)\text{ exists, and }\notag\\
\omega(y,0) &= \omega_0&&~~\text{ on }~~(0,r(0)).\notag
\end{align}
Here $\omega_0:(0,r(0))\rightarrow \R$ generates the initial graph $\omega_0
\in C^2((0,r(0)))$ that also satisfies the Neumann boundary condition
$\IP{{\nu}_{\omega_0}}{{\nu}_{\Sigma}}= 0$ at $r(0)$.

Note that in this representation the graph direction for $\omega_\Sigma$ is
perpendicular to the graph direction for $\omega$.
(Contrast with \cite{vmwheeler2012rotsym}.)
The two graphs share the same axis of revolution.
Examples of this include graphs evolving inside a vertical catenoid neck or
inside the hole of a vertical unduloid.

\section{Existence and prerequisites}

Global existence of solutions to \eqref{Neumannproblem} under restrictions on
$\Sigma$ was shown in \cite {wheeler2016mean} by obtaining uniform $C^1$
estimates.
The problem \eqref{Neumannproblem} is a quasilinear second-order PDE on a
time-dependent domain with a Neumann boundary condition.
The change in domain can be calculated (see \eqref{rprime}) and depends only on
$\omega_\Sigma$, $\omega'$, and $\omega''$.
The local unique existence of a solution in this setting is standard and has
been discussed in detail in \cite{thesisvulcanov,vmwheeler2012rotsym}.

The boundary condition $\IP{{\nu}_\omega}{{\nu}_{\Sigma}} = 0$ can be written in a
simpler way if we take into account the fact that we are working with two
graph functions.
The outer normal to $\omega$ is given by
\[
{\nu}_{\omega}=\frac{1}{\sqrt{1+(\frac{d\omega}{dy})^2}}\Big(-\frac{d\omega}{dy},1\Big)\,.
\]
For the unit normal to $\omega_\Sigma$ we need to rotate and translate the axes.
We find
\[
{\nu}_{\Sigma}=\frac{1}{\sqrt{1+(\frac{d{\omega}_{\Sigma}}{dz})^2}}\Big(1,-\frac{d{\omega}_{\Sigma}}{dz}\Big)\,.
\]
This transforms the Neumann boundary condition into
\begin{align}
\frac{d \omega}{dy}(r(t),t)=-\frac{d {\omega}_{\Sigma}}{dz}(\omega(r(t),t))~~
\text{  for all   } t\in[0,T),\label{Neumanncondition}
\end{align}
and gives us the following uniform boundary gradient estimate for $\omega$ by an upper bound on the gradient of $\omega_\Sigma$.
\begin{lem}[Uniform boundary gradient estimates]
Let $\omega_\Sigma$ and ${\omega}_0$ be defined as above.
Assume \eqref{Sigma_graph}.

Then
\begin{align*}
\bigg|\frac{d \omega}{dy}(r(t),t)\bigg| \leq  \sqrt{\frac{1}{C_\Sigma}-1}
\end{align*}
for all $t\in[0,T)$.
\label{gradientestimates}
\end{lem}

\begin{rmk}
On the free Neumann boundary, the rotational symmetry of the solution prevents tilt behaviour.
This occurs when the normal to the graph becomes parallel to the vector field of rotation for $\Sigma$.
This behaviour is explained in much greater detail in \cite{thesisvulcanov} and it is present in many situations of free
boundary problems \cite{andrewswheeler}, thus the need to use the rotationally symmetry in constructing the barriers
needed to show the elliptic results.
\end{rmk}

The derivative of the boundary point $r(t)$ is computed as follows.
Since
\begin{align*}
r(t) = \omega_\Sigma(\omega(r(t),t))\,,
\end{align*}
we calculate
\begin{align*}
r'(t) = \frac{d\omega_\Sigma}{dz}\bigg(\frac{\partial\omega}{\partial t} + \frac{d\omega}{dy} r'(t)\bigg)\,.
\end{align*}
Substituting in the boundary condition \eqref{Neumanncondition} and $\frac{\partial\omega}{\partial t}= -Hv$
yields
\begin{align}
r'(t) = -\frac{H}{v}\frac{d\omega_\Sigma}{dz},
\label{rprime}
\end{align}
where we have once again denoted $v=\sqrt{1+(\frac{d\omega}{dy})^2}$, and at
the boundary $v=\sqrt{1+(\frac{d\omega_\Sigma}{dz})^2}$.

The norm squared of the second fundamental form and mean curvature in terms of
the profile curve $\omega$ are expressed as in the following lemma.

\begin{lem}
For a rotationally symmetric hypersurface generated by the rotation of a graph
function $\omega$ about an axis perpendicular to the graph direction, the norm
squared of the second fundamental form and mean curvature are given by the
formulae
\begin{align*}
|A|^2 &= \frac{1}{(1+(\frac{d\omega}{dy})^2)^3} \Big(\frac{d^2\omega}{dy^2}\Big)^2
          + \frac{1}{1+(\frac{d\omega}{dy})^2}\frac{1}{y^2}\Big(\frac{d\omega}{dy}\Big)^2\notag\\
H &= -\frac{1}{\sqrt{1+(\frac{d\omega}{dy})^2}^3} \frac{d^2\omega}{dy^2} - \frac{1}{\sqrt{1+(\frac{d\omega}{dy})^2}}\frac{1}{y}\frac{d\omega}{dy}.
\end{align*}
\label{secondff}
\end{lem}

\begin{figure}
\begin{tabular}{cc}
\includegraphics[trim=4cm 10cm 4cm 10cm,clip=true,width=6cm,height=4cm]{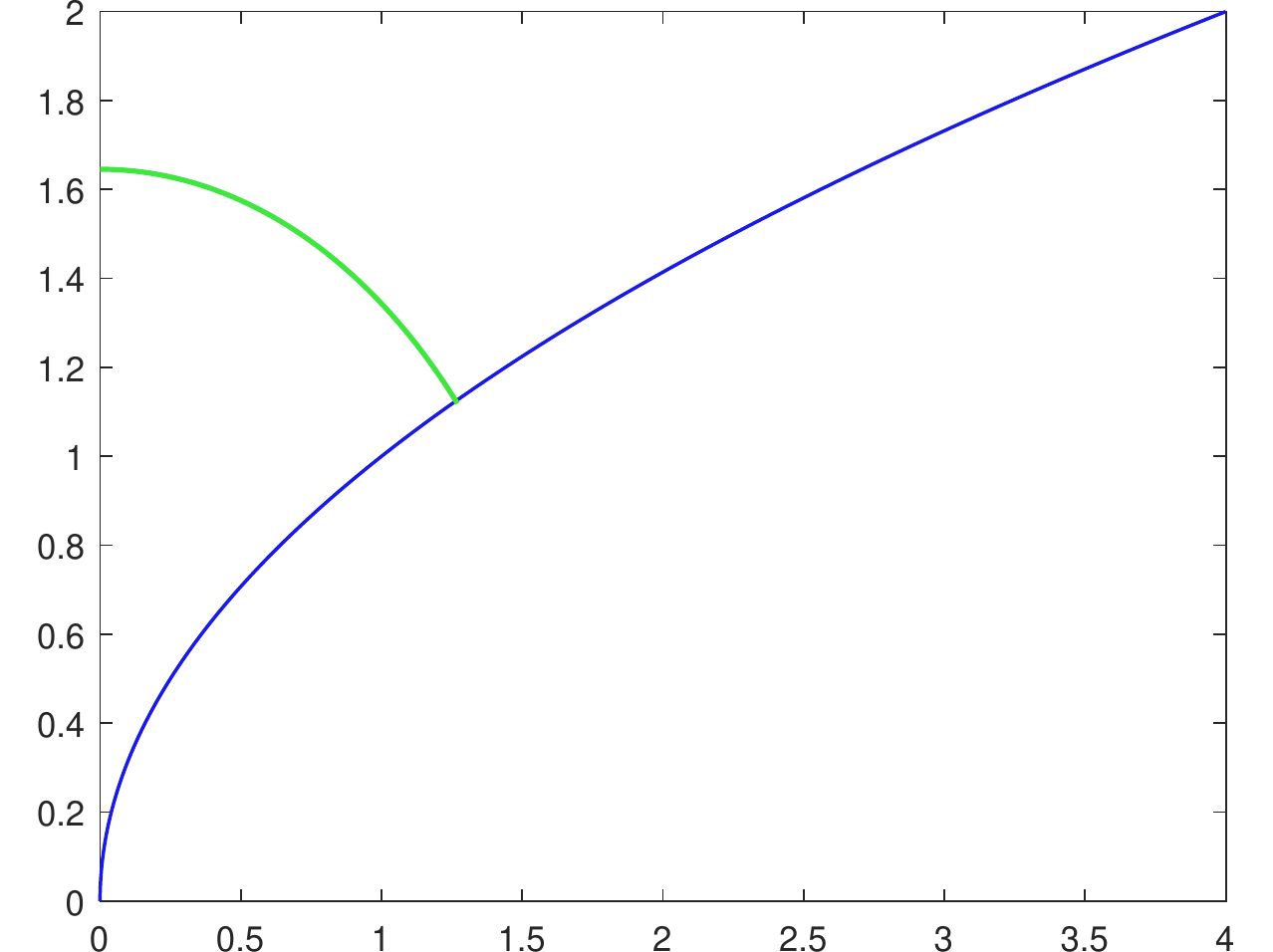}
&
\includegraphics[trim=4cm 10cm 4cm 10cm,clip=true,width=6cm,height=4cm]{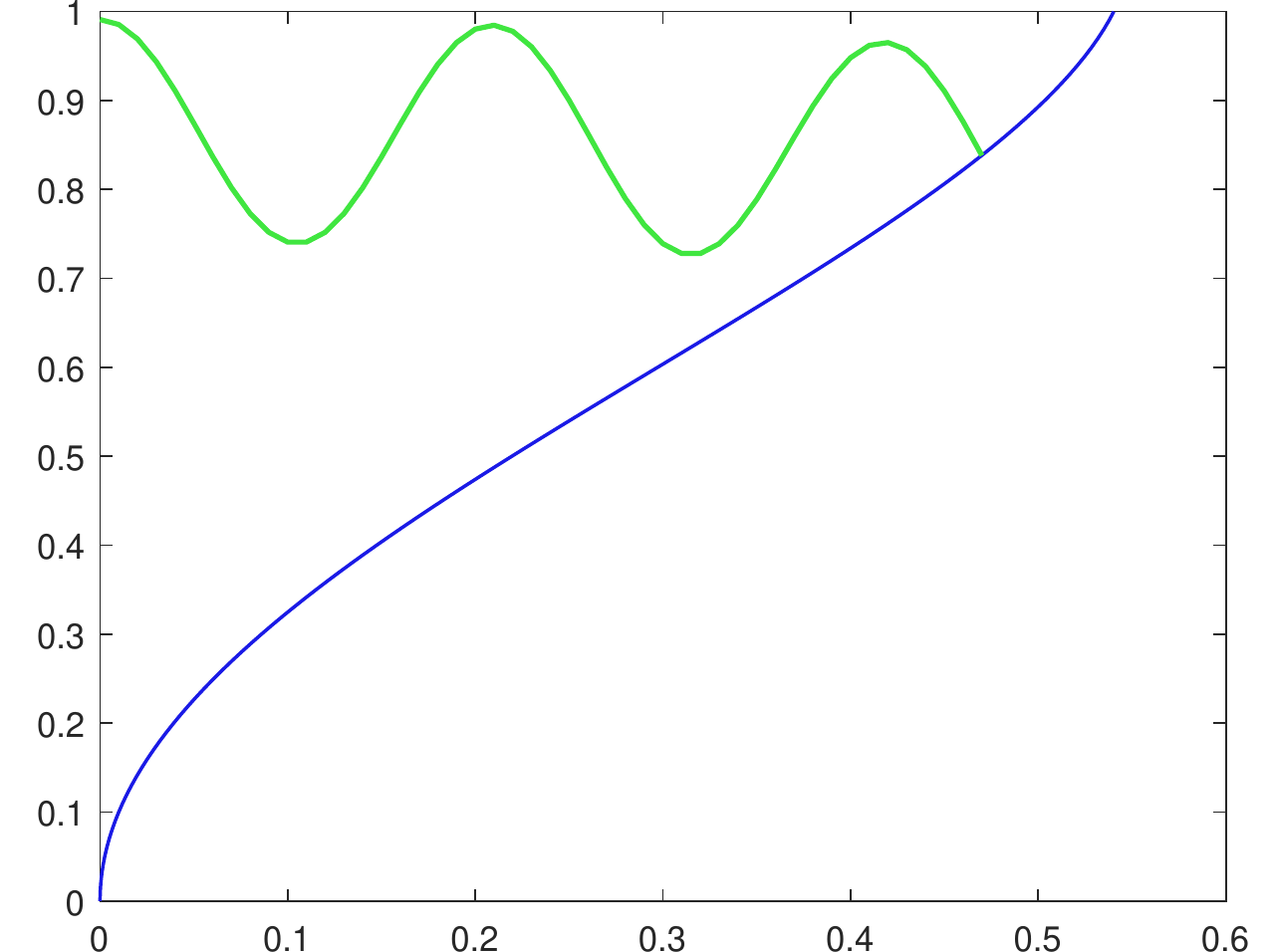}
\end{tabular}
  \caption{Examples of initial data that evolve toward finite-time singularities.}
  \label{fig:singularity}
\label{Fig2}
\end{figure}

\section{Type $2$ singularities}

We treat the case when the support hypersurface $\Sigma$ pinches on its axis of rotation; that is, there
exists one or more points $z^*$ such that $\omega_\Sigma(z^*) = 0$. We do not require that $\Sigma$ is smooth at those points so examples of such support hypersurfaces include cones,
parabolae or hypersurfaces that form cusps at the rotation axis.

Let us recall Definition \ref{pinchingcylinder}.

\begin{udefn}[Pinching cylinder]
Let $\omega_\Sigma:Oz\rightarrow[0,\infty)$ be a continuous function.
Assume that $\omega_\Sigma$ is smooth outside finitely many points $P =
\{w_1,\ldots,w_{n_p}\}$, where $\omega_\Sigma(w_i) = 0$; that is,
$\omega_\Sigma\in C^\infty_{loc}(Oz\setminus P)$.
Assume that there exists a compact set $K \supset P$ such that
\begin{equation*}
z\frac{d\omega_\Sigma}{dz}(z) > 0\quad\text{ for all }z\in Oz\setminus K\,.
\end{equation*}
The function $\omega_\Sigma$ generates a smooth rotationally symmetric disconnected hypersurface
$F_\Sigma:\Sigma\rightarrow\R^{n+1}$, where $\Sigma$ is the disjoint union of $n_p+1$ cylinders.
We term the support hypersurface $F_\Sigma$ a \emph{pinching cylinder}.
\end{udefn}

\begin{rmk}
Although we require that $\omega_\Sigma$ be only continuous on $\R$, it may
pinch and be smooth (or analytic) everywhere on $Oz$. This is the case if
$\omega_\Sigma$ is a non-negative polynomial in $z$ with zeros; for example,
\[
\omega_\Sigma(z) = (z-2)^2(z+2)^2\,.
\]
\end{rmk}

We also recall Theorem \ref{thmsingularities}.

\begin{uthm}[Flow in pinching cylinders \cite{wheeler2016mean}]
Let $\Sigma$ be a pinching cylinder as in Definition \ref{pinchingcylinder} with $n_p=1$.
Let $w_1 = z^* = 0$.
Assume that for all $z\in Oz$
\begin{align*}
\IP{{\nu}_{\Sigma}(z)}{e_1} > C_\Sigma \geq 0
\end{align*}
where $C_\Sigma$ is a global constant and ${\nu}_{\Sigma}$ is the normal to
${\omega}_{\Sigma}$.
The graph condition \eqref{Sigma_graph} is understood as limits from above and
below at points in $P$.

Suppose that for all $z\in Oz\setminus\{0\}$,
\begin{equation*}
z\frac{d\omega_\Sigma}{dz}(z) > 0\,.
\end{equation*}
Then the maximal time $T$ of existence for any solution $\omega:D(t)\times[0,T)\rightarrow\R$ to \eqref{Neumannproblem}
is finite.
The hypersurfaces $F:D^n\times[0,T)\rightarrow\R^{n+1}$ generated by $\omega$ contract as $t\rightarrow T$ to the point
$(0,0)$.
\end{uthm}

\begin{rmk}[Non-rotational initial data]
Any initially bounded mean curvature flow with free boundary, irrespective of
symmetry or topological properties, exists at most for finite time when
supported on a pinching cylinder as in Theorem \ref{thmsingularities}.
This is because so long as the initial immersion is bounded, we may always
construct a rotationally symmetric graphical solution such that the initial
immersion lies between this solution and the pinchoff point $(0,z^*)$.
The flow generated by this pair of initial data remain disjoint by the
comparison principle, and as the rotationally symmetric solution contracts to a
point in finite time, the flow of immersions must either develop a curvature
singularity in finite time or contract to the same point (and possibly remain
regular while doing so).

\end{rmk}

Singularities are classified as Type 0, 1 or 2.
In \cite{wheeler2016mean}, we classified most cases as being Type 1 or better, see Figure \ref{Fig3} for an illustration of the prototypical Type 1 singularity.
Type 0 singularities are not curvature singularities at all but a loss of domain.
The cases that allow us to do this are when the gradient of $\omega_\Sigma$ is
bounded.  This includes cones and cusps.

\begin{figure}
\includegraphics[trim=4cm 10cm 4cm 10cm,clip=true,width=8cm,height=6cm]{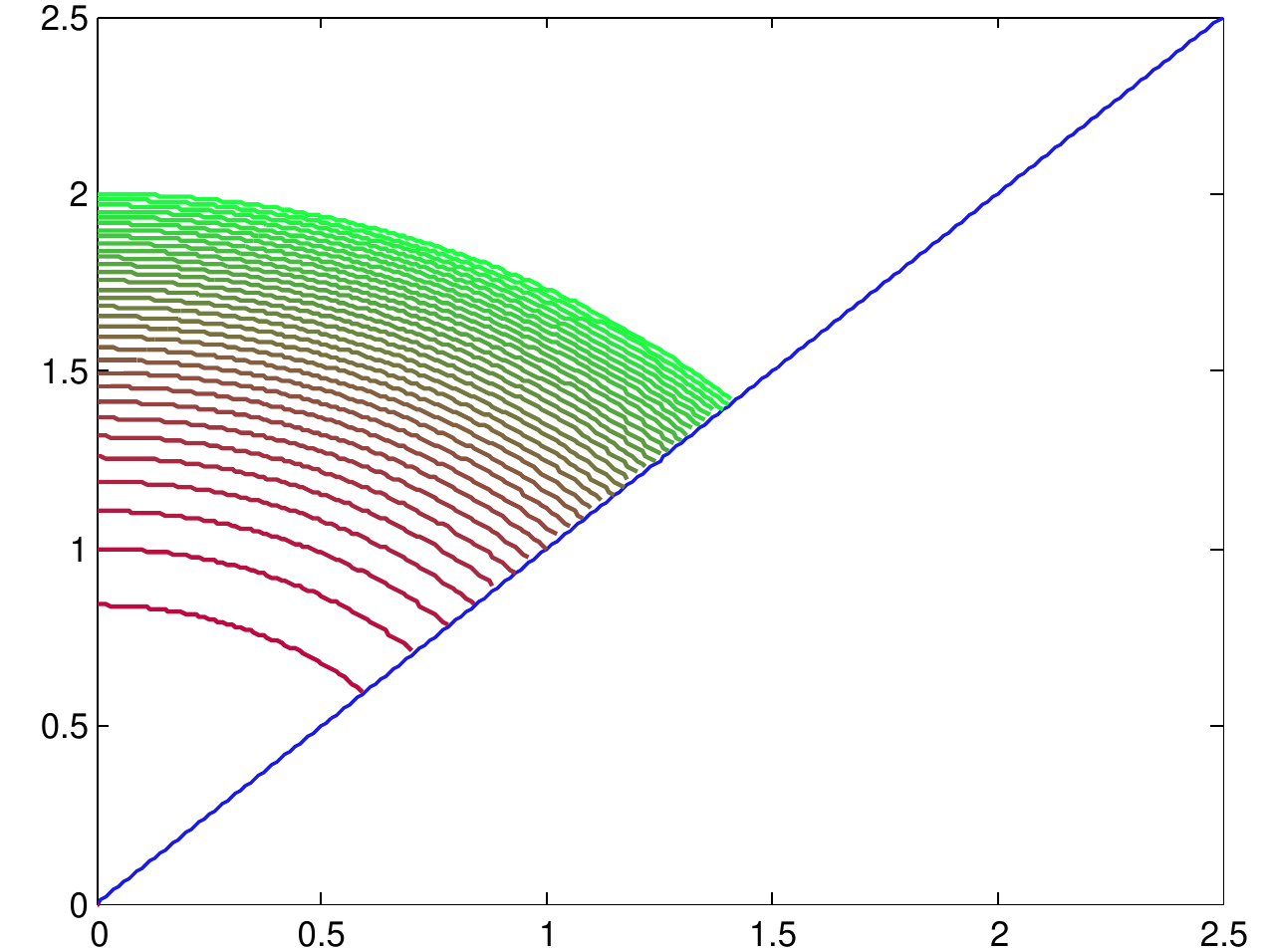}
\caption{The mean curvature flow with free boundary on a standard cone with
spherical initial data shrinks to a Type 1 singularity in finite time.
The curves show a solution with initial radius 2, at times $t = \frac{\tilde{t}}{(n-1)}$ where $\tilde{t} \in\{0,0.25,0.50,0.75,\ldots,8.0\}$. Time moves forward as green becomes red. The solution begins moving slowly, however quickly speeds up. The amount of time to move from one leaf to the next is equal. The final dark red leaf shrinks to the origin in $\frac{1}{4(n-1)}$ units of time.}
\label{Fig3}
\end{figure}

\begin{defn}[Singularities]
Let $F:D^n\times[0,T)\rightarrow\R^{n+1}$ be a mean curvature flow with free boundary supported on a pinching cylinder.
If there exists an $\varepsilon>0$ such that for all $t\in(T-\varepsilon,T)$
\begin{itemize}
\item the second fundamental form is uniformly bounded, that is,
\[
|A|^2(x,t) \le C < \infty\,,
\]
then we say the singularity is Type 0;
\item the second fundamental form is uniformly controlled under parabolic rescaling, that is,
\[
|A|^2(x,t) \le \frac{C}{T-t}\,,
\]
then we say the singularity is Type 1;
\item neither of the previous two cases apply, we say the singularity is Type 2.
\end{itemize}
\end{defn}

\begin{thm}[Type 1 singularities]
Let $\omega_\Sigma$ and $\omega_0$ be as in Theorem \ref{thmsingularities}.
If there exist two constants $0<C_1<\infty$ and $C_2<\infty$ such that for $z$
sufficiently close to $z^*$ we have:
\begin{itemize}
\item Conical pinchoff
\begin{align*}
C_1 \le \bigg|\frac{d\omega_\Sigma}{dz}(z^*)\bigg| \le C_2,
\end{align*}
then the singularity from Theorem \ref{thmsingularities} is Type 1;
\item Polynomial pinchoff
\begin{align*}
C_1|\omega_\Sigma(z)|^{\sigma} \le \bigg|\frac{d\omega_\Sigma}{dz}(z)\bigg| \le C_2|\omega_\Sigma(z)|^{\sigma}\,,
\end{align*}
for $\sigma<1$, then the singularity from Theorem \ref{thmsingularities} is
Type 1, and in particular there exist $\hat{C}_1, \hat{C}_2$ such that for $t$
sufficiently close to $T$ we have
\[
\frac{\hat{C}_1}{T-t}
\le
|A|^2(x,t)
\le
\frac{\hat{C}_2}{T-t}\,.
\]
\end{itemize}
\label{thmtypeI}
\end{thm}

\begin{rmk}
Conical pinchoff is a special case of polynomial pinchoff.
For polynomial pinchoff, it isn't possible to satisfy all conditions of the
theorem for $\sigma\ge1$.
For $\sigma>0$, the pinchoff is \emph{convex} and for $\sigma<0$ the pinchoff
is \emph{concave}.
These names come from the following examples:
\[
\omega_\Sigma(z) = z^\alpha
\]
satisfies $\omega_\Sigma'(z) = \alpha \omega_\Sigma^{1-\frac1\alpha}(z)$.
Therefore $\alpha > 1$ corresponds to $\sigma \in (0,1)$ and $\alpha < 1$ corresponds
to $\sigma < 0$.
Clearly all asymptotically polynomial pinchoffs are allowed by the condition $\sigma < 1$.
Concave pinchoff is related to the singularity resulting from mean curvature
flow with free boundary supported in the sphere, studied by Stahl
\cite{stahl1996res}.
\end{rmk}

\begin{thm}[Type 0 singularities]
Let $\omega_\Sigma:Oz\rightarrow\R$ be the profile curve of a rotationally
symmetric hypersurface satisfying \eqref{Sigma_graph} and
\[
\lim_{z\rightarrow\infty}\omega_\Sigma(z) = 0\,,\quad
\bigg|\frac{d\omega_\Sigma}{dz}(z)\bigg| \le C|\omega_\Sigma|^{1+\sigma}(z)\,,\quad \sigma>0\,.
\]
Then the maximal time of existence for any solution $\omega:D(t)\times[0,T)\rightarrow\R$ to \eqref{Neumannproblem}
satisfies $T = \infty$.
The hypersurfaces $F:D^n\times[0,T)\rightarrow\R^{n+1}$ generated by $\omega$ satisfy
\[
||A||_\infty^2(t) \rightarrow \alpha_0\quad\text{as $t\rightarrow\infty$}\,,
\]
and so either
\begin{itemize}
\item $F(D^n,t)$ converges smoothly to a flat disk; or
\item Modulo translation, $F(D^n,t)$ converges to a flat point, that is, a singularity of Type 0.
\end{itemize}
\label{thmtype0}
\end{thm}

\begin{rmk}
Examples of support hypersurfaces with profile curves satsifying the conditions
of Theorem \ref{thmtype0} include exponentials and reciprocal polynomials, such as
\[
\omega_\Sigma(z) = e^{-z}
\]
and a (monotone) mollification of
\[
\omega_\Sigma(z) =
\begin{cases}
\frac1z\,,\quad\text{for }z>1
\\
-z+2\,,\quad\text{for }z\le1\,.
\end{cases}
\]
\end{rmk}

Theorem \ref{thmtypeII} brings Type 2 singularities into the picture.
We now give its proof.

\begin{proof}[Proof of Theorem \ref{thmtypeII}]
As the gradient $\omega_\Sigma$ is uniformly bounded, we obtain uniform bounds on the gradient of $\omega$ as in \cite{wheeler2016mean}.
Height bounds also follow exactly as in \cite{wheeler2016mean}.
Since the mean curvature flow of graphs is a quasilinear second order parabolic PDE, this implies uniform bounds on higher derivatives, in particular there exists a $C_4\in(0,\infty)$ depending only on $\omega_0$ and $\omega_\Sigma$ such that
\begin{align*}
\bigg|\frac{d^2\omega}{dy^2}\bigg| \leq C_4\,.
\end{align*}
Note that the singularity guaranteed by Theorem \ref{thmsingularities} occurs
in spite of these estimates given by standard machinery of parabolic PDE as the
domain shrinks away to nothing.
As it does so, the curvature may explode at various rates, governed primarily
by the blowup rate of $\frac{1}{\omega}$, a quantity that the standard theory
does nothing to control.

This also implies that the graphs are smooth and exist up to the time of
shrinking to the pinching point, remaining graphs. The maximum principle
implies that the second fundamental form is bounded everywhere by its values on
the boundary, see \cite{ecker1989mce}.
We estimate
\begin{align}
\sup_{y\in[0,r(t)]} |A|^2(y,t) &= |A|^2(r(t),t) \notag\\
 &= \frac{1}{(1+(\frac{d\omega}{dy})^2)^3}
	\Big(\frac{d^2\omega}{dy^2}\Big)^2(r(t),t)
          + \frac{1}{1+(\frac{d\omega}{dy})^2}\frac{1}{r(t)^2}\Big(\frac{d\omega}{dy}\Big)^2(r(t),t)\notag\\
&\geq C_\Sigma\frac{1}{r(t)^2}\Big(\frac{d\omega}{dy}\Big)^2(r(t),t)\notag\\
&=C_\Sigma\frac{1}{\omega_\Sigma(z)^2}\Big(\frac{d\omega_\Sigma}{dz}\Big)^2(z)\notag\\
&\geq C_\Sigma C_1\frac{1}{z^{2\delta}}\notag\\
&=C_\Sigma C_1\frac{1}{\omega^{2\delta}(r(t),t)}\label{first}\,,
\end{align}
where we have used the lower bound on the ratio of $\omega_\Sigma$ to its
gradient for all $z$ sufficiently close to $0$ and Lemma
\ref{gradientestimates} to estimate
\[
\frac{1}{1+(\frac{d\omega}{dy})^2}
\ge \frac{1}{1 + (1/C_{\Sigma} - 1)}
 =  C_\Sigma\,.
\]

We note that $z=\omega(r(t),t)$ so we determine the speed at which $\omega$
decreases on the boundary to find out more information on the asymptotic
behaviour of $|A|^2$.

We use the parabolic evolution for $\omega$ and the time evolution
for $r(t)$ from \eqref{rprime} to compute
\begin{align*}
\frac{d\omega}{dt} &= =\partial_t \omega + \frac{d\omega}{dy}r'(t)
	= - \frac{H}{\sqrt{1+\big(\frac{d\omega}{dy}\big)^2}}\,.
\end{align*}
Substituting the formula for the mean curvature in Lemma \ref{secondff}, we obtain
\begin{align*}
\frac{d\omega}{dt} =
\frac{1}{\sqrt{1+\big(\frac{d\omega}{dy}\big)^2}^4}\frac{d^2\omega}{dy^2}+
\frac{1}{\sqrt{1+\big(\frac{d\omega}{dy}\big)^2}^2}\frac{1}{y}\frac{d\omega}{dy}.
\end{align*}
Using the Neumann boundary condition \eqref{Neumanncondition} and the upper
bound on the quotient of $\omega_\Sigma$ and its derivative we obtain
\begin{align*}
\frac{d\omega}{dt} =
\frac{1}{\sqrt{1+\big(\frac{d\omega}{dy}\big)^2}^4}\frac{d^2\omega}{dy^2} -
\frac{1}{\sqrt{1+\big(\frac{d\omega}{dy}\big)^2}^2}\frac{1}{\omega_\Sigma}\frac{d\omega_\Sigma}{dz}
\geq  -C_4 - \frac{C_2}{z^\alpha}\,.
\end{align*}
We have also used that the second derivative is bounded by $C_4$ and thus there
exists a $t^*$ such that for $t^*<t<T$ we may multiply by $z^\alpha$ to absorb
the first term into the second above:

\begin{align*}
\frac{d\omega}{dt}\omega^\alpha \geq  - C_4 \omega^\alpha  -  C_2 \geq  - C_5,
\end{align*}
for some positive $C_5>0$.

Integrating from $t<T$ to $T$ and using the fact that $\omega(T)=0$ we find
\begin{align*}
\frac{1}{\omega^{\alpha+1}}  \geq  \frac{1}{C_5(\alpha+1)}\frac{1}{T-t},
\end{align*}
for all $t\geq t^*$.
Substituting this into \eqref{first} we obtain the following bound for the second fundamental form
\begin{align*}
\sup_{y\in[0,r(t)]} |A|^2(y,t)
 \geq C\frac{1}{(T-t)^{\frac{2\delta}{\alpha+1}}},
\end{align*}
for all $t\in (t^*,T)$ and
$C={C_\Sigma C_1}\big(C_5(\alpha+1)\big)^{-\frac{2\delta}{\alpha+1}}$,
that is, the singularity is Type 2, given that $\frac{2\delta}{\alpha+1}>1$.
\end{proof}

We complete the paper by giving the proof of Theorem \ref{thmtype02}.

\begin{proof}[Proof of Theorem \ref{thmtype02}]

Let us assume that the singularity produced by Theorem \ref{thmsingularities}
is type 0.
Then there exists a $\tilde{C} < \infty$ such that for all $t\leq T$ we have
\[
|A|^2(x,t)\leq \tilde{C}\,.
\]
Now by Lemma \ref{secondff} this implies
\[
\frac{1}{(1+(\frac{d\omega}{dy})^2)^3} \Big(\frac{d^2\omega}{dy^2}\Big)^2
  +
  \frac{1}{1+(\frac{d\omega}{dy})^2}\frac{1}{y^2}\Big(\frac{d\omega}{dy}\Big)^2
  \le \tilde{C}\,.
\]
Now as noted in the proof of Theorem \ref{thmtypeII} the gradient of $\omega$ is uniformly bounded and so there exists a constant $0<C<\infty$ such that for $z$
sufficiently close to $0$ we have:
\begin{align*}
	\bigg|\frac{\frac{d\omega_\Sigma}{dz}(z)}{\omega_\Sigma(z)}\bigg| \le C\,,
\end{align*}
and $\omega_\Sigma(0) = 0$.
Furthermore, by \eqref{conelike} there exists an $\varepsilon>0$ such that for
all $z\in(0,\varepsilon)$
\[
	\frac{d\omega_\Sigma}{dz}(z) > 0\,.
\]
Therefore we have
\[
\frac{d\omega_\Sigma}{dz}(z) \le C|\omega_\Sigma(z)| = C\omega_\Sigma(z)\,,\quad z\in(0,\varepsilon)\,.
\]
The second equality follows by smoothness of the generated hypersurface
$\Sigma$; if $\omega_\Sigma$ crossed the axis of revolution then $\Sigma$ would
be singular.

Now the differential form of Gr\"onwall's inequality applies to give
\[
	\omega_\Sigma(z) \le \omega_\Sigma(0)e^{\int_0^z\tilde{C}\,dw} = 0\,.
\]
Therefore $\omega_\Sigma(z) = 0$ for all $z\in[0,\varepsilon/2]$, which is a
contradiction with a variety of assumptions, for example $n_p = 1$ and
\eqref{conelike}.
\end{proof}

\section*{acknowledgements}

The authors are partially supported by Australian Research Council Discovery
grant DP150100375 at the University of Wollongong.

\bibliographystyle{plain}
\bibliography{mbib}

\begin{thebibliography}{10}

\bibitem{andrewswheeler}
B.~Andrews and V.-M. Wheeler.
\newblock Counterexamples to graph preservation in mean curvature flow with
  free boundary.
\newblock {\em preprint}, 2016.

\bibitem{buckland2005mcf}
J.A. Buckland.
\newblock {Mean curvature flow with free boundary on smooth hypersurfaces}.
\newblock {\em J. Reine Angew. Math.}, 586:71--90, 2005.

\bibitem{ecker2004rtm}
K.~Ecker.
\newblock {\em {Regularity Theory for Mean Curvature Flow}}.
\newblock Birkhauser, 2004.

\bibitem{ecker1989mce}
K.~Ecker and G.~Huisken.
\newblock {Mean curvature evolution of entire graphs}.
\newblock {\em Ann. of Math. (2)}, 130(2):453--471, 1989.

\bibitem{ecker1991ieh}
K.~Ecker and G.~Huisken.
\newblock {Interior estimates for hypersurfaces moving by mean curvature}.
\newblock {\em Invent. Math.}, 105(1):547--569, 1991.

\bibitem{edelen2014convexity}
Nick Edelen.
\newblock Convexity estimates for mean curvature flow with free boundary.
\newblock {\em Advances in Mathematics}, 294:1--36, 2016.

\bibitem{freire2010mean}
A.~Freire.
\newblock Mean curvature motion of graphs with constant contact angle at a free
  boundary.
\newblock {\em Analysis \& PDE}, 3(4):359--407, 2010.

\bibitem{koeller2007singular}
A.~Koeller.
\newblock {\em {On the Singularity Sets of Minimal Surfaces and a Mean
  Curvature Flow}}.
\newblock PhD thesis, Freie Universit\"at Berlin, 2007.

\bibitem{lambert2012constant}
B.~Lambert.
\newblock The constant angle problem for mean curvature flow inside rotational
  tori.
\newblock {\em arXiv preprint arXiv:1207.4422}, 2012.

\bibitem{lambert2014perpendicular}
B.~Lambert.
\newblock The perpendicular neumann problem for mean curvature flow with a
  timelike cone boundary condition.
\newblock {\em Transactions of the American Mathematical Society},
  366(7):3373--3388, 2014.

\bibitem{thesisstahl}
A.~Stahl.
\newblock {\em {\"Uber den mittleren Kr\"ummungsfluss mit Neumannrandwerten auf
  glatten Hyperfl\"achen}}.
\newblock PhD thesis, Fachbereich Mathematik, Eberhard-Karls-Universit\"at,
  T\"uebingen, Germany, 1994.

\bibitem{stahl1996convergence}
A.~Stahl.
\newblock Convergence of solutions to the mean curvature flow with a neumann
  boundary condition.
\newblock {\em Calc. Var. Partial Differential Equations}, 4(5):421--441, 1996.

\bibitem{stahl1996res}
A.~Stahl.
\newblock {Regularity estimates for solutions to the mean curvature flow with a
  neumann boundary condition}.
\newblock {\em Calc. Var. Partial Differential Equations}, 4(4):385--407, 1996.

\bibitem{thesisvulcanov}
V.-M. Vulcanov.
\newblock {\em {Mean curvature flow of graphs with free boundaries}}.
\newblock PhD thesis, Freie Universit\"at, Fachbereich Mathematik und
  Informatik, Berlin, Germany, 2011.

\bibitem{wheeler2014mean}
G.~Wheeler and V.-M. Wheeler.
\newblock Mean curvature flow with free boundary outside a hypersphere.
\newblock {\em arXiv preprint arXiv:1405.7774}, 2014.

\bibitem{wheeler2014meanhyperplane}
V.-M. Wheeler.
\newblock Mean curvature flow of entire graphs in a half-space with a free
  boundary.
\newblock {\em Journal f{\"u}r die reine und angewandte Mathematik (Crelles
  Journal)}, 2014(690):115--131, 2014.

\bibitem{vmwheeler2012rotsym}
V.-M. Wheeler.
\newblock Non-parametric radially symmetric mean curvature flow with a free
  boundary.
\newblock {\em Math. Z.}, 276(1-2):281--298, 2014.

\bibitem{wheeler2016mean}
V.-M. Wheeler.
\newblock Mean curvature flow with free boundary in embedded cylinders or cones
  and uniqueness results for minimal hypersurfaces.
\newblock {\em arXiv preprint arXiv:1605.09098}, 2016.

\end{thebibliography}

\end{document}